\documentclass[runningheads]{llncs}
\usepackage[T1]{fontenc}
\usepackage{amsmath}
\usepackage{mathtools}
\usepackage{amssymb}
\usepackage{graphicx}
\allowdisplaybreaks
\newcommand{ \supp }{ \operatorname{supp} }

\title{Ponzi schemes on coarse spaces with uniform measure}
\author{Shunsuke MIYAUCHI}
\authorrunning{S. Miyauchi}
\institute{Graduate School of Mathematical Sciences, The University of Tokyo, Komaba, Meguro, Tokyo, 153, Japan, \email{miyashun@ms.u-tokyo.ac.jp}}

\begin{document}

\maketitle              % typeset the header of the contribution
\begin{center}
{\itshape Dedicated to Professor Toshiyuki Kobayashi with great admire for his groundbreaking achievements}
\end{center}
\begin{abstract}
Ponzi schemes, defined by Block-Weinberger(1992) and Roe(2003), give a characterization of amenability from the viewpoint of coarse geometry. We consider measures in coarse spaces, and propose a reformulation of Ponzi schemes with measures.

\end{abstract}
\section{Introduction}
Block-Weinberger~\cite{BW92} defined Ponzi schemes in metric spaces, and Roe~\cite{Roe03} extended them to coarse spaces. They proved for metric spaces the equivalence between the non-existence of a Ponzi scheme and amenability. Since Ponzi schemes are well-defined up to coarse equivalence, they give a characterization of amenability from the view point of coarse geometry. Roe~\cite{Roe03} extended the notion of amenability to general coarse spaces with Ponzi schemes according to the above consideration. In other words, a coarse space is called non-amenable if it has a Ponzi scheme.

In this paper, we consider a \lq measure\rq \ in a coarse space, which we call a uniform measure (Definition~\ref{def:um}), and propose a reformulation of Ponzi schemes with a uniform measure $\mu$, which we call a $\mu$-PS (Ponzi scheme with the uniform measure $\mu$, Definition~\ref{def:muPS}). Ponzi schemes in \cite{Roe03} are defined only with a coarse structure, using counting method. On the other hand, $\mu$-PSs are defined with a coarse structure, a $\sigma$-algebra and a uniform measure $\mu$. 

The main result is the following theorem.
\begin{theorem} {\rm(See Theorem~\ref{thm:muPS})}
Let $(X,\mathcal{E})$ be a measurable coarse space, $\mu$ be a uniform measure, and the measure space $(X,\mu)$ be $\sigma$-finite.
\begin{enumerate}
\item Assume $X$ has a $\mu$-PS and a quasi lattice $\Lambda$. Then, $X$ has a Ponzi scheme.
\item Assume $X$ has a Ponzi scheme $\theta$ such that $\supp \theta \subset S \times S$ for some uniform locally finite subset $S$, and there exists a measurable controlled set $E$ which is constant on S with respect to $\mu$. Then, $X$ has a $\mu$-PS. 
\end{enumerate}
\end{theorem}
This theorem gives a relationship between Ponzi schemes and what we call $\mu$-PSs. By using this theorem, we prove the relationship between the amenability of a discrete group $\Gamma$ and the existence of a $\mu$-PS on a space $X$ with respect to $\Gamma$-invariant measure $\mu$ as follows: 

\begin{corollary} {\rm (See Corollary~\ref{cor:gpactmuPS})}
Let a discrete group $\Gamma$ act properly and cocompactly on a locally compact second-countable Hausdorff space $X$, and $\mu$ be a $\Gamma$-invariant non zero regular measure on $X$. Then $\Gamma$ is non-amenable if and only if the coarse space $(X,\mathcal{E}^X_{\Gamma})$ has a $\mu$-PS.
\end{corollary}

\section{Preliminaries and Notations}

In this section, we review the definition of coarse spaces and some notions in measure theory. We refer \cite{Roe03} and \cite{Co13} for more details.

Let $X$ be a set. We use the following notations:
\begin{itemize}
\item Let $E \subset X\times X$. Let $E^{T}$ denote the transpose set $\{ (x^{\prime},x): (x,x^\prime) \in E\}$.
\item Let $E_1, E_2 \subset X\times X$. Let $E_1\circ E_2$ denote the product of the sets $ \{ (x,z):(x,y)\in E_1, (y,z) \in E_2 \text{ for some  }y\in X\}$.
\item Let $E \subset X\times X, K \subset X$. Let $E[K]$ denote the set $\{ y\in X :(x,y)\in E \text{ for some }x\in K\}$. In particular, we write $E[\{x\}]:= E_x$.
\item Let $\Delta_X$ denote the diagonal set $\{ (x,x):  x \in X\}$.
\item Let $\varphi\colon X \to Y$, $E \subset X \times X$, and $F \subset Y \times Y$. Let $(\varphi \times \varphi)_* (E)$ denote the image $\{ (\varphi(x), \varphi(y)) : (x,y) \in E \}$, and let $(\varphi \times \varphi)^*(F)$ denote the preimage $\{ (x,y) : (\varphi(x), \varphi(y)) \in F \}$.
\end{itemize}

\begin{definition}{\rm (\cite[Def.\ 2.3]{Roe03}) }
A subset $\mathcal{E} \subset \mathcal{P}(X\times X)$ is a coarse structure if it satisfies the following properties: 
\begin{enumerate}
\item $\Delta_X \in \mathcal{E}$.
\item If $E \subset X\times X, E^\prime \in \mathcal{E},$ and $E \subset E^\prime,$ then $E \in \mathcal{E}$ (i.e., $\mathcal{E}$ is closed under the operation of subset).
\item If $E_1,E_2 \in \mathcal{E}$, then $E_1\cup E_2 \in \mathcal{E}$ (i.e., $\mathcal{E}$ is closed under the operation of finite union).
\item If $E_1,E_2 \in \mathcal{E}$, then $E_1\circ E_2 \in \mathcal{E}$ (i.e., $\mathcal{E}$ is closed under the operation of product).
\item If $E\in \mathcal{E}$, then $E^T \in \mathcal{E}$ (i.e., $\mathcal{E}$ is closed under transposing).
\end{enumerate}
An element $E$ in a coarse structure $\mathcal{E}$ is called a controlled set. A set equipped with a coarse structure $(X,\mathcal{E})$ is called a coarse space.

\end{definition}
\begin{definition}{\rm (\cite[Prop.\ 2.16]{Roe03})}
A subset $B$ in a coarse space $X$ is called a bounded set if there exists a controlled set $E$ and $x\in X$ such that $B \subset E_x$.
\end{definition}
In order to consider a map between coarse spaces, we introduce the following notions.
\begin{definition}\rm{ (\cite[Def.\ 2.14]{Roe03}) }
Let $(X,\mathcal{E})$ be a coarse space and $S$ be a set. Two maps $\varphi_1, \varphi_2\colon S \to X$ are close if $\{(\varphi_1(s), \varphi_2(s)): s \in S\}$ is a controlled set. Then, we write $\varphi_1 \sim \varphi_2$.
\end{definition}

\begin{definition}\rm{(\cite[Def.\ 2.21]{Roe03}\cite{Zava19})}  
Let $\varphi\colon(X,\mathcal{E}^X) \to (Y, \mathcal{E}^Y)$ be a map between coarse spaces.
\begin{itemize}
\item A map $\varphi$ is bornologous if $(\varphi \times \varphi)_*(E) \in \mathcal{E}^Y$ for any controlled set $E \in \mathcal{E}^X$. 
\item A map $\varphi$ is proper if $\varphi^{-1}(B)$ is bounded for any bounded set $B$ in $Y$.
\item A map $\varphi$ is a coarse map if it is proper  and bornologous.
\item A coarse map $\varphi$ is coarsely equivalent if there exists a coarse map $\psi\colon Y \to X$ such that $\psi \circ \varphi \sim id_X$ and $\varphi \circ \psi \sim id_Y$. Two spaces $X$ and $Y$ are coarsely equivalent if there exists a coarsely equivalent map between $X$ and $Y$.
\item A map $\varphi$ is effectively proper if $(\varphi \times \varphi)^*(F)\in \mathcal{E}^X $ for any controlled set $F \in \mathcal{E}^Y$.
\item A map $\varphi$ is coarsely surjective if there exists $F_0 \in \mathcal{E}^Y$ such that $Y =F_0[\varphi(X)]$.
\end{itemize}
\end{definition}

There is a well-known equivalent condition of the coarsely equivalence of a map. 

\begin{fact}\label{fact:ceq}{\rm (\cite{Zava19})}
The map $\varphi$ is coarsely equivalent if and only if it is bournologous, effectively proper, and coarsely surjective.
\end{fact}

We give some examples of coarse space below.
\begin{example}{\rm (\cite[Ex.\ 2.5]{Roe03})}\label{ex:bcstr}
Let $(X,d)$ be a metric space. Let $\mathcal{E}^d := \{ E \subset E_r : \text{for some } r >0\}$, where $E_r := \{ (x,y) : d(x,y) \leq r\}$. The collection $\mathcal{E}^d$ is a coarse structure, which is called the bounded structure associated to the metric $d$.
\end{example}
\begin{example}{\rm (\cite{AD12})}
Let $G$ be a topological group, and let $\mathcal{E}_R^G := \{ E \subset E_C : C\text{ is compact.} \} $, where $E_C := \{ (x,y) : x^{-1}y \in C \} $. The collection $\mathcal{E}_R^G$ is a coarse structure, which is called a right $G$ compact structure.
\end{example}

The next example was introduced in \cite{BDM08}. 

Assume a topological group $G$ acts on a topological space $X$. For a subset $S$ in $X$, we write $G_S := \{ g\in G : gS \cap S \neq \emptyset \}$. The action is proper if $G_S$ is compact for any compact subset $S$. The action is cocompact if there exists a compact subset $S_0 \subset X$ such that $ X = GS_0 = \{ gs : g \in G,s \in S_0 \}$. 
\begin{example}{\rm(\cite{BDM08})}
Assume a locally compact Hausdorff group $G$ acts on a locally compact Hausdorff space $X$ properly and cocompactly. The collection $\mathcal{E}^X_G$ is defined by $\mathcal{E}^X_G := \{ E \subset E_{G,C} : \text{for some compact subset }C \subset X \}$, where $E_{G,C} := \bigcup_{g \in G} gC \times gC $. The collection $\mathcal{E}^X_G$ is a coarse structure, and the following map $f\colon G \to X$ is coarsely equivalent for a fixed $x_0 \in X$:
\[ f\colon (G,\mathcal{E}^G_R) \ni g \mapsto gx_0 \in (X, \mathcal{E}^X_G)\]
\end{example}

We review some facts in measure theory. Let $X$ be a topological space. Let $\mathcal{B}(X)$ denote the collection of all Borel sets, that is the minimum $\sigma$-algebra which contains all open sets.
Let $\mathcal{O}(X\times X)$ denote the collection of all open sets in $X \times X$ in the product topology, that is the one generated by sets $\{U\times V : U,V \text{ are open} \}$. Let $\mathcal{B}(X\times X)$ denote the minimum $\sigma$-algebra which contains all elements in $\mathcal{O}(X \times X)$.
Let $\mathcal{B}$ be a $\sigma$-algebra. Let $\mathcal{B} \times \mathcal{B}$ denote the minimum $\sigma$-algebra which contains all products of measurable sets. 

\begin{fact}\label{fact:mth}{\rm  (\cite[Lem.\ 5.1.2, Prop.\ 7.6.2, Prop.\ 4.5.1]{Co13}) }
\begin{enumerate}
\item Let $\mathcal{B}$ be a $\sigma$-algebra. If $E \in \mathcal{B}\times \mathcal{B}$, then $E_x \in \mathcal{B}$ for all $x \in X$.
\item If $X$ be a locally compact second-countable Hausdorff space, then 
\[\mathcal{B}(X) \times \mathcal{B}(X) = \mathcal{B}(X\times X).\]
\item Let a measure space $(X,\mu)$ be $\sigma$-finite. Then, we have the following isometry:
\[ \left(L^1(X,\mu)\right)^* \cong L^{\infty}(X,\mu) . \]
\end{enumerate}
\end{fact}

\section{Uniform measure}
\subsection{Measurablity of controlled sets}
In this subsection, we review the measurability of controlled sets defined in \cite{Wi21}, and define a similar but different measurability.  

Let $X$ be a set, $\mathcal{B}$ be a $\sigma$-algebra on $X$, and $\mathcal{E}$ be a coarse structure on $X$.

\begin{definition}\label{def:mblecstrW}{\rm(\cite{Wi21})}
A controlled set $E$ is measurable if $E[U]$ is a measurable set for all measurable sets $U$ in $X$.
\end{definition}

In contrast, we consider the following.
\begin{definition}\label{def:mblecstr}
\begin{enumerate}
\item A subset $E$ in $X\times X$ is a measurable controlled set if $E \in \mathcal{E}\cap(\mathcal{B\times B})$.
\item The coarse structure $\mathcal{E}$ is measurable if for all $E \in \mathcal{E}$ there exists $\tilde{E} \in \mathcal{E}\cap (\mathcal{B}\times \mathcal{B})$ such that it contains $E$.

\end{enumerate}
\end{definition}

We explain the difference between the two notions of measurability.
\begin{proposition}\label{prop:diff}
\begin{enumerate}
\item There exists a coarse space $X$ with a $\sigma$-algebra $\mathcal{B}$ and a controlled set $E$ such that it is measurable in the sense of Definition~\ref{def:mblecstrW} but not measurable in the sense of Definition~\ref{def:mblecstr}.
\item Consider the coarse structure of $\mathbb{R}$ induced by the Euclidean metric, and $\sigma$-algebra $\mathcal{B}(\mathbb{R})$. There exists a controlled set $E^{\prime}$ in $\mathbb{R\times R}$ such that it is measurable in the sense of Definition~\ref{def:mblecstr} but not measurable in the sense of Definition~\ref{def:mblecstrW}.
\end{enumerate}
\end{proposition}

We need the following facts for the proof.

\begin{fact}{\rm(\cite[Chap.\ 5]{Co13})}\label{fact:diag}
If the cardinality of $X$ is greater than that of the continuum, then the diagonal set $\Delta$ in $X\times X$ is not Borel measurable.
\end{fact}
\begin{fact}{\rm(\cite{S1917})}\label{fact:analynotmble}
There exists a Borel measurable set $A$ in $\mathbb{R}^2$ such that its projection onto $\mathbb{R}$ is not Borel measurable.
\end{fact}
{\itshape Proof of Proposition~\ref{prop:diff}.} {\itshape (1)} 
Let $X$ be a coarse space whose cardinality is larger than the continuum.
By Fact~\ref{fact:diag}, the diagonal set $\Delta$ is not measurable in the sense of Definition~\ref{def:mblecstr}, but it is measurable in the sense of Definition~\ref{def:mblecstrW} since $\Delta[U]=U$.

{\itshape (2)} 
Define $E^{\prime} \coloneqq A \cap ([0,1]\times[0,1])$. It is controlled and measurable in the sense of Definition~\ref{def:mblecstr}, but not measurable in the sense of Definition~\ref{def:mblecstrW} since $E^{\prime}[\mathbb{R}] = \mathrm{Proj}(E^{\prime})$ is not Borel measurable in $\mathbb{R}$ by Fact~\ref{fact:analynotmble}.
\qed

\vskip\baselineskip

An advantage of Definition~\ref{def:mblecstrW} is that the composition of measurable controlled sets is again measurable. An analogue statement does not always hold if we adopt Definition~\ref{def:mblecstr}. 
On the other hand, an advantage of Definition~\ref{def:mblecstr} is that it is compatible to the notion of measurability of product spaces, and that the fiber of singleton $E_x$ is automatically measurable by Fact~\ref{fact:mth}(1).  An analogue statement does not hold if we adopt Definition~\ref{def:mblecstrW}.

We adopt the measurability of Definition~\ref{def:mblecstr} in this paper because we consider $E$-neighborhoods of a point in many cases, and we can overcome its disadvantage by considering a measurable controlled set containing the composition of controlled sets. 

\subsection{Uniform measure}
We review the notions of measure uniformity defined in \cite{Wi21}.
 \begin{definition}{\rm(\cite{Wi21})}\label{def:um}
 A measure $\mu$ on $(X,\mathcal{B})$ is uniform with respect to the measurable coarse space $\mathcal{E}$ if the following evaluation 
 \[ \sup_{x\in X} \mu(E_x) < \infty\]
 holds for any measurable controlled set $E \in \mathcal{E} \cap (\mathcal{B}\times \mathcal{B})$. 
 \end{definition}

\subsection{Examples of measurable coarse structures and uniform measures}

We give some examples of measurable coarse structures and uniform measures.
\begin{example}
Let $(X,d)$ be a separable metric space. Then $\mathcal{E}^d$ is measurable. Indeed, a set $E_r$ is closed in $X \times X$ for all $r>0$, thus $E_r \in \mathcal{B}(X\times X) = \mathcal{B}(X)\times \mathcal{B}(X)$ by Fact~\ref{fact:mth}(2).

A measure $\mu$ on $(X,\mathcal{B}(X))$ is uniform with respect to $\mathcal{E}^d$ if and only if
\[ \sup_{x\in X} \mu \left( \overline{B_r}(x)\right) < \infty\]
holds for all $r>0$, where the closed ball with radius $r$ is denoted by $\overline{B_r}(x)$.
\end{example}

\begin{example}
Let $G$ be a locally compact second-countable Hausdorff group. Then, the coarse structure $\mathcal{E}^G_R$ is measurable. Indeed, a set $E_C$ is closed in $G\times G$ for all compact subsets $C$ in $G$. If $\mu$ is the left Haar measure, then $\mu$ is uniform. Indeed, we have 
\[ \mu \left( (E_C)_x\right) = \mu (xC) = \mu(C) < \infty \]
for all $x \in G$ and all compact subsets $C$ (note that the left Haar measure $\mu$ is regular).
\end{example}

\begin{example}
Let $G$ be a locally compact group and $X$ be a locally compact second-countable Hausdorff space, and suppose that $G$ acts on $X$ properly and cocompactly. Then, we have
\begin{enumerate}
\item The coarse structure $\mathcal{E}^X_G$ is measurable.
\item If a measure $\mu$ on $X$ is $G$ invariant and regular, then $\mu$ is uniform with respect to $\mathcal{E}^X_G$.
\end{enumerate}
\end{example}
\begin{proof}
(1) We prove that $G$ is a Lindel\"of space. Indeed, we have a increasing sequence of compact subsets $K_1 \subset K_2 \subset \cdots$ such that $X = \bigcup_{i=1}^{\infty} K_i$, since $X$ is locally compact and second-countable, in particular, it is $\sigma$-compact. We can check $G = \bigcup_{i=1}^{\infty} G_{K_i}$, and each $G_{K_i}$ is compact since the $G$-action is proper. Therefore, $G$ is $\sigma$-compact, in particular, Lindel\"of.

Fix a relatively compact open neighborhood $U_0$ of the unit in $G$. Since $G = \bigcup_{g \in G}g U_0$, there exists a countable subset $\{ g_i\}_{i=1}^{\infty}$ such that $G = \bigcup_{i=1}^{\infty} g_i U_0$.

For any compact subset $C$ in $G$, we have
\begin{eqnarray*}
E_{G,C} &=& \bigcup_{g\in G} gC \times gC \\
&\subset & \bigcup_{i=1}^{\infty} g_iU_0 C \times g_iU_0 C \quad \text{(this term is denoted by } \tilde{E_{G,C}} ) \\
&\subset& \bigcup_{g \in G} g\bar{U_0}C \times g\bar{U_0}C = E_{G,\bar{U_0}C}
\end{eqnarray*}
The set $\tilde{E_{G,C}}$ is contained in $\mathcal{B}(X) \times \mathcal{B}(X)$ and is a controlled set since $\tilde{E_{G,C}} \subset E_{G,\bar{U_0}C} \in \mathcal{E}^X_G$. Therefore, $\mathcal{E}^X_G$ is measurable.

(2) For any compact subset $C \subset X$ satisfying $X = G C$ and $x \in X$, there exists $g_x \in G$ such that $x \in g_x C$. If $y \in (E_{G,C})_x$, i.e., $(x,y) \in E_{G,C}$, then there exists $g \in G$ such that $x,y \in gC$. Since $gC \cap g_x C \neq \emptyset$, we have $g_x^{-1}g \in G_C$. Thus, we have
\[y \in gC \subset g_x (g_x^{-1} g)C \subset g_x G_C C .\]
Hence, we have $(E_{G,C})_x \subset g_x G_C C$.  By the definition of $\mathcal{E}^X_G$, for any $E \in \mathcal{E}^X_G$, there exists a compact set $C \subset X$ such that $E \subset E_{G,C}$. Therefore, we have 
\[ \mu(E_x) \leq \mu(g_x G_C C) = \mu(G_C C) < \infty\]
for all $x \in X$. The second equality follows since $\mu$ is $G$-invariant, and the last inequality follows by the regularity of $\mu$.
\qed
\end{proof}

\section{An analogue of Ponzi scheme for coarse spaces with uniform measures}

\subsection{Definition of $\mu$-PS and Ponzi scheme}

In this subsection, we review the definition of  Ponzi scheme, and we define $\mu$-PS, which is the main object of this paper. 
\begin{definition}{\rm (\label{def:PS}\cite[Def.\ 3.34]{Roe03})}
A uniform 1-chain $\theta$ is a Ponzi scheme if its boundary $\partial \theta$ is effective.
\end{definition}

\begin{definition}\label{def:muPS}
A $\mu$-1 chain $c$ is a $\mu$-PS if its boundary $\partial c$ is effective.
\end{definition}

\subsection{Review of the notions in the definition of Ponzi scheme}
We recall the notions used in Definition~\ref{def:PS} as below. Let $(X,\mathcal{E})$ be a coarse space. 
\begin{definition}{\rm (\cite[Def.\ 3.25]{Roe03})}
A subset $S$ in $X$ is uniform locally finite if 
\[ \sup_{x \in X} \# (E_x \cap S) < \infty \] 
holds for all controlled set $E \in \mathcal{E}$.
\end{definition}

\begin{definition}{\rm (\cite[Def.\ 3.27 - Def.\ 3.32]{Roe03})} 
\begin{enumerate}
\item A function $\phi\colon X \to \bbbr$ is a uniform 0-chain if it is a bounded function and there exists a uniform locally finite subset $S$ in $X$ such that $\supp \phi$ is contained in it. The collection of uniform 0-chains is denoted by $C^u_0(X)$. A uniform 0-chain $\phi$ is effective if $\phi \geq 0$, and there exists a controlled set $E$ such that the following evaluation
\[ \sum_{s\in E_x \cap S} \phi(s) \geq 1 \]
holds for all $x \in X$.
\item A function $\theta\colon X \times X \to \bbbr$ is a uniform 1-chain if it is a bounded function and there exists a uniform locally finite subset $S$ and a controlled set $E_{\theta} \in \mathcal{E}$ such that $\supp \theta \subset (S \times S) \cap E_{\theta}$. The collection of uniform 1-chains is denoted by $C^u_1(X)$.
\end{enumerate}
\end{definition}

Let us define the boundary map for a uniform 1-chain by
\[ \partial \theta (s) := \sum_{s^{\prime} \in S} \theta(s^{\prime},s ) - \sum_{s^{\prime} \in S} \theta (s,s^{\prime}) .\]
We can define a Ponzi scheme by the following fact:
\begin{fact}{\rm (\cite[Chap.\ 3]{Roe03})}
The boundary map $\partial$ is a well-defined map from $C^u_1 (X)$ to $C^u_0 (X)$.
\end{fact}

\subsection{The notions in the definition of $\mu$-PS}

We define the notions in Definition~\ref{def:muPS}. Let $(X,\mathcal{B})$ be a measurable space, and $\mathcal{E}$ be a measurable coarse structure, $\mu$ be a uniform measure on $X$ with respect to $\mathcal{E}$. Moreover, we assume the measure space $(X,\mu)$ is $\sigma$-finite.

\begin{definition}
\begin{enumerate}
\item A function $f\colon X \to \bbbr$ is a $\mu$-0 chain if it is an essentially bounded with respect to $\mu$, i.e., $f \in L^{\infty}(X,\mu)$. The collection of $\mu$-0 chains is denoted by $C_0(X,\mu)$. A $\mu$-0 chain $f$ is effective if $f \geq 0 $ a.e. $\mu$ and there exists $E \in \mathcal{E} \cap (\mathcal{B}\times \mathcal{B})$ such that $\int_{E_{x_0}} f d\mu \geq 1$ for all $x_0 \in X$.
\item A function $c\colon X \times X \to \bbbr$ is a $\mu$-1 chain if it is essentially bounded with respect to $\mu \otimes \mu$ and there exists a measurable controlled set $E_c \in \mathcal{E}\cap (\mathcal{B}\times \mathcal{B})$ such that $\int_{X\times X \setminus E_c}|c| d \mu \otimes \mu = 0$. Just as $\mu$-0 chain  case, the collection of $\mu$-1 chains is denoted by $C_1(X,\mu)$.
\end{enumerate}
\end{definition}

In order to define a $\mu$-PS, we consider the boundary map from $C_1(X,\mu)$ to $C_0(X,\mu)$.
\begin{proposition}
The boundary map $\partial\colon C_1(X,\mu) \to C_0(X,\mu)$ defined as follows is well-defined :
\[ \partial c (x) = \int_X c(y,x) d\mu(y) - \int_X c(x,y) d\mu(y)  \]
for $c \in C_1(X,\mu)$.
\end{proposition}
\begin{proof}
We have 
\[ \int_{X \setminus (E_c^T)_x} |c(y,x)| d\mu(y) = \int _{X\setminus (E_c)_x} |c(x,y)| d\mu(y) = 0\]
for $x \in X$ a.e. $\mu$ by the assumption of $\mu$-1 chain. 

Thus, for $x \in X$ a.e. $\mu$, we have
\begin{eqnarray*}
|\partial c (x)| 
&\leq & \int_{(E_c^T)_x} |c(y,x)| d\mu(y) + \int_{(E_c)_x} |c(x,y)| d\mu(y) \\
&\leq & ||c||_{L^{\infty}(X\times X)} \mu \left( (E_c^T)_x \right) + ||c||_{L^{\infty}(X\times X)} \mu \left( (E_c)_x \right) < \infty
\end{eqnarray*}
uniformly as $x \in X$. Therefore, $\partial c \in C_0(X,\mu)$.
\qed
\end{proof}

\subsection{Some remarks on Ponzi schemes and $\mu$-PS}
We do some remarks about the similarities and differences between Ponzi schemes and $\mu$-PSs.
\begin{remark}
We can consider a Ponzi scheme with a uniform locally finite support $S \times S$ as a $\#|_S$-PS, where 
\[ \#|_S (A) := \# (A \cap S)\]
for $A \subset X$.
\end{remark}

\begin{remark}
In order to define $\mu$-PSs, we need to assume that the measure space $(X,\mu)$ is $\sigma$-finite, for it confirms several properties, for example, the uniqueness of the product measure $\mu \otimes \mu$. However, the countability of support of Ponzi schemes is not assumed in its definition (see Definition~\ref{def:PS}).
\end{remark}

\subsection{An example of $\mu$-PS}
In this subsection, we give an example of $\mu$-PS.
\begin{example}\label{ex:muPS}
    Consider the Poincare disk model of the hyperbolic plane, that is $D^2 := \{ (x,y) \in \bbbr^2 : x^2 + y^2 < 1 \}$ with the Poincare metric $ds^2 = \frac{dx^2+dy^2}{(1-x^2-y^2)^2}$, and define the measure $\mu$ on $D^2$ by $\mu:= \frac{4dxdy}{(1-x^2-y^2)^2}$. Let $d$ denote the induced metric. Then, the measurable coarse space $(D^2,\mathcal{E}^{d})$ has a $\mu$-PS.
\end{example}
For the proof of this statement, we review some facts in hyperbolic geometry.  We refers \cite{And05} and \cite{Rat19}.

\begin{fact}{\rm (\cite[Thm.\ 3.5.3]{Rat19} \cite[Thm.\ 5.16, Exer.\ 4.2]{And05})}\label{fact:hyp}
\begin{enumerate}
\item (The First Law of Cosines) If $\alpha, \beta, \gamma$  are the angles of hyperbolic triangle and $a, b, c$ are the length of the opposite sides, then
\[ \cos\gamma = \frac{\cosh a \cosh b - \cosh c}{\sinh a \sinh b}.\]
\item If $\alpha, \beta, \gamma$ are the angles of hyperbolic triangle $\triangle$, then
\[ \mu(\triangle) = \pi -(\alpha + \beta + \gamma).\]
\item For any $\theta \in [0,2\pi)$ and  $r\geq 0$, we  have
\[ d(0,re^{i\theta})= \log \left( \frac{1+r}{1-r}\right).\]
\end{enumerate}
\end{fact}

\noindent {\itshape Proof of Example~\ref{ex:muPS}.} Define the function $c\colon D^2\times D^2 \to \bbbr$ by
\begin{equation*}
 c(z,z^{\prime}):= \left\{ \begin{array}{lll}
    1 &\quad &( d(z,z^{\prime})\leq 1 \ \& \ d(z,0)\geq d(z^{\prime},0))\\
    0 &\quad &(\text{otherwise}) .\\
    \end{array} \right.
\end{equation*}
The function $c$ is bounded, in particular essentially bounded $a.e.~\mu$, and the support of $c$ is contained in $E_1$ (see Example~\ref{ex:bcstr}), thus $c$ is a $\mu$-1 chain. We prove that there exists $\varepsilon > 0$ such that 
\[ \partial c (z) \geq \varepsilon \]
holds for all $z \in D^2$. Then, by multiplying a positive scalar, the $\mu$-1 chain $c$ is a $\mu$-PS, since all closed ball in $D^2$ with a same positive radius has the same positive volume. Define $D(z) := \{ z^{\prime} \in D^2 : d(z,z^{\prime}) \leq 1 , d(0,z)>d(0,z^{\prime} )\} $ for $z \in D^2$. We have  
\begin{align*}
\partial c(z) &= \int_{D^2}c(z^{\prime},z)d\mu(z^{\prime}) - \int_{D^2} c(z,z^{\prime}) d\mu(z^{\prime}) \\
&= \int_{\overline{B_1}(z)\setminus D(z)} d\mu(z^{\prime}) - \int_{D(z)} d\mu(z^{\prime})\\
&= \mu\left( \overline{B_1}(z) \right) - 2 \mu(D(z))  .
\end{align*}
The volume of the closed ball $\overline{B_1}(z)$ can be calculated as follows:
\begin{align*}
\mu\left( \overline{B_1}(z)\right) &= \mu\left( \overline{B_1}(0) \right) \\
&= \iint_{\overline{B_1}(0)} \frac{4dxdy}{(1-x^2-y^2)^2} \\
&= \int_{0}^{2\pi}\int_0^{1/2} \frac{4}{(1-\tanh^2{t})^2} \frac{\tanh{t}}{\cosh^2{t}}dt d\theta \\
&= 2\pi (\cosh{1}-1) .\tag{*}
\end{align*}
We changed the variable $(x,y)$ to $(t,\theta)$ in the third equality by
\[ x = \tanh{t}\cos\theta, y = \tanh{t}\sin\theta.\]
We write $z = \tanh{\left( \frac{r}{2}\right)} e^{i\theta}$. By Fact~\ref{fact:hyp}(3), 
\[ d(0,z) = \log \left( \frac{1-\tanh{r/2}}{1+\tanh{r/2}}\right) = r .\]
In order to calculate the volume of the area $D(z)$, we consider the following cases,
{\itshape Case~1} : $0<r\leq\frac{1}{2}$, and {\itshape Case~2} : $ \frac{1}{2}<r$.

\noindent{\itshape Case~1.} In this case, we have $D(z) = \overline{B_r}(0)$, since the closed ball $\overline{B_r}(0)$ is contained in $\overline{B_1}(z)$. In the same fashion as (*), we have $\mu\left( D(z)\right) = 2\pi(\cosh{r}-1)$. Hence, we have 
\begin{align*}
\partial c(z) &= 2\pi (\cosh{1}-1) -4\pi (\cosh{r} - 1) \\
&\geq 2\pi (\cosh{1} - 1) -4\pi\left(\cosh{\frac{1}{2}} -1\right) > 0 .
\end{align*}
The second inequality holds since the function as $r$ is monotonically increasing.
\begin{figure}[htbp]
\begin{center}
\includegraphics[width=25mm]{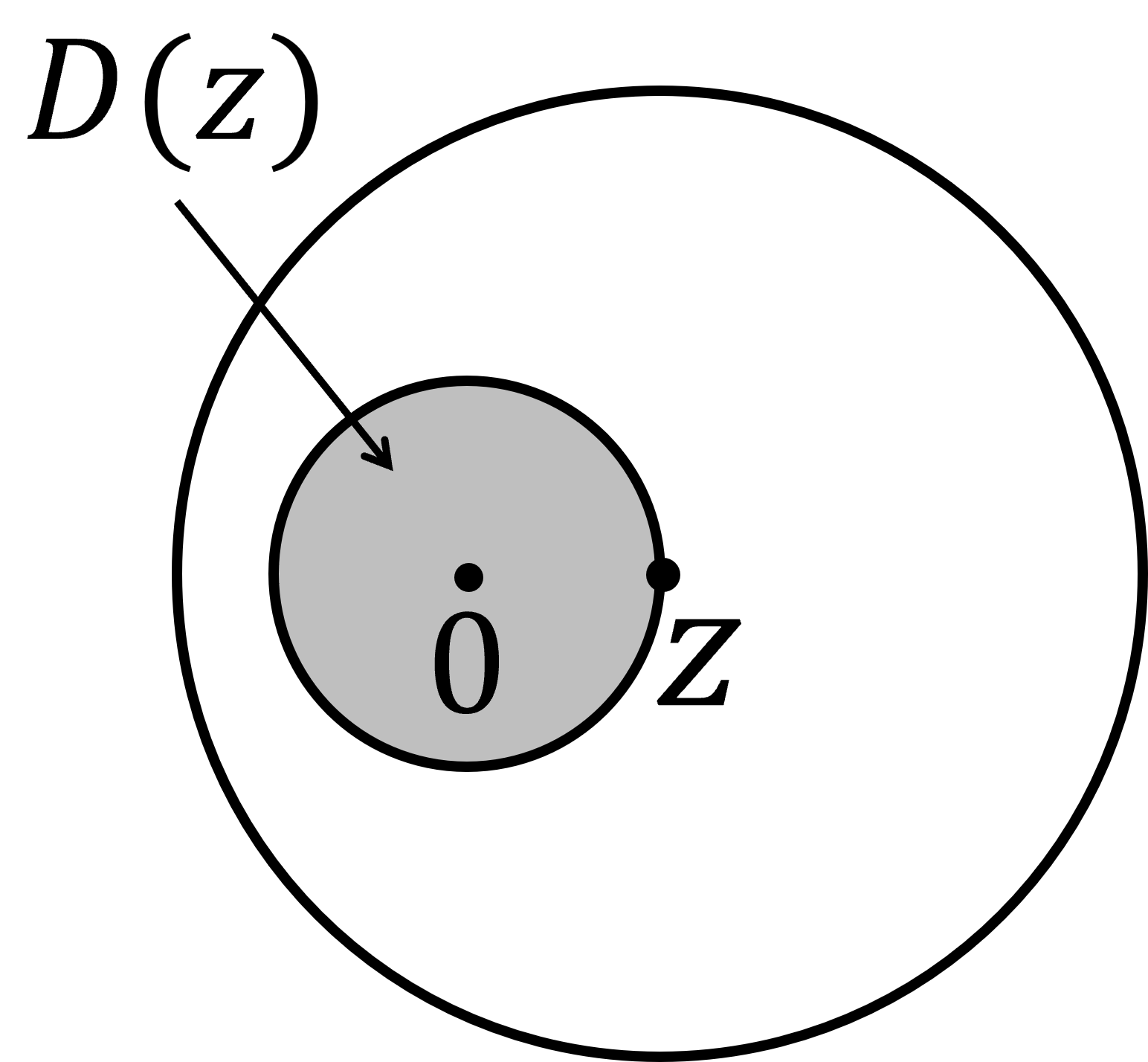}
\caption{{\itshape Case~1.} in the proof of {\itshape Example~\ref{ex:muPS}.}}
\end{center}
\end{figure}

\noindent{\itshape Case~2.} In order to consider this case, we introduce the following notations:
\begin{itemize}
\item Let $\triangle abc$ denote the triangle with vertices $a,b,c$.
\item Let $T(abcd)$ denote the tetragon with vertices $a,b,c,d$.
\item Let $S(o;ab)$ denote the sector with an arc $ab$ and a center $o$.
\item Let $C_r(o)$ denote the circle with a center $o$ and radius $r$.
\end{itemize}
Let $p^{+}$ and $p^{-}$ denote the intersections between $C_r(0)$ and $C_1(z)$,  $\alpha$ denote the angle between the two geodesics $op^{+}$ and $oz$, $\beta$ denote that between $zp^{+}$ and $zo$. 
\begin{figure}[htbp]
\begin{center}
\begin{minipage}{40mm}
\includegraphics[width=30mm]{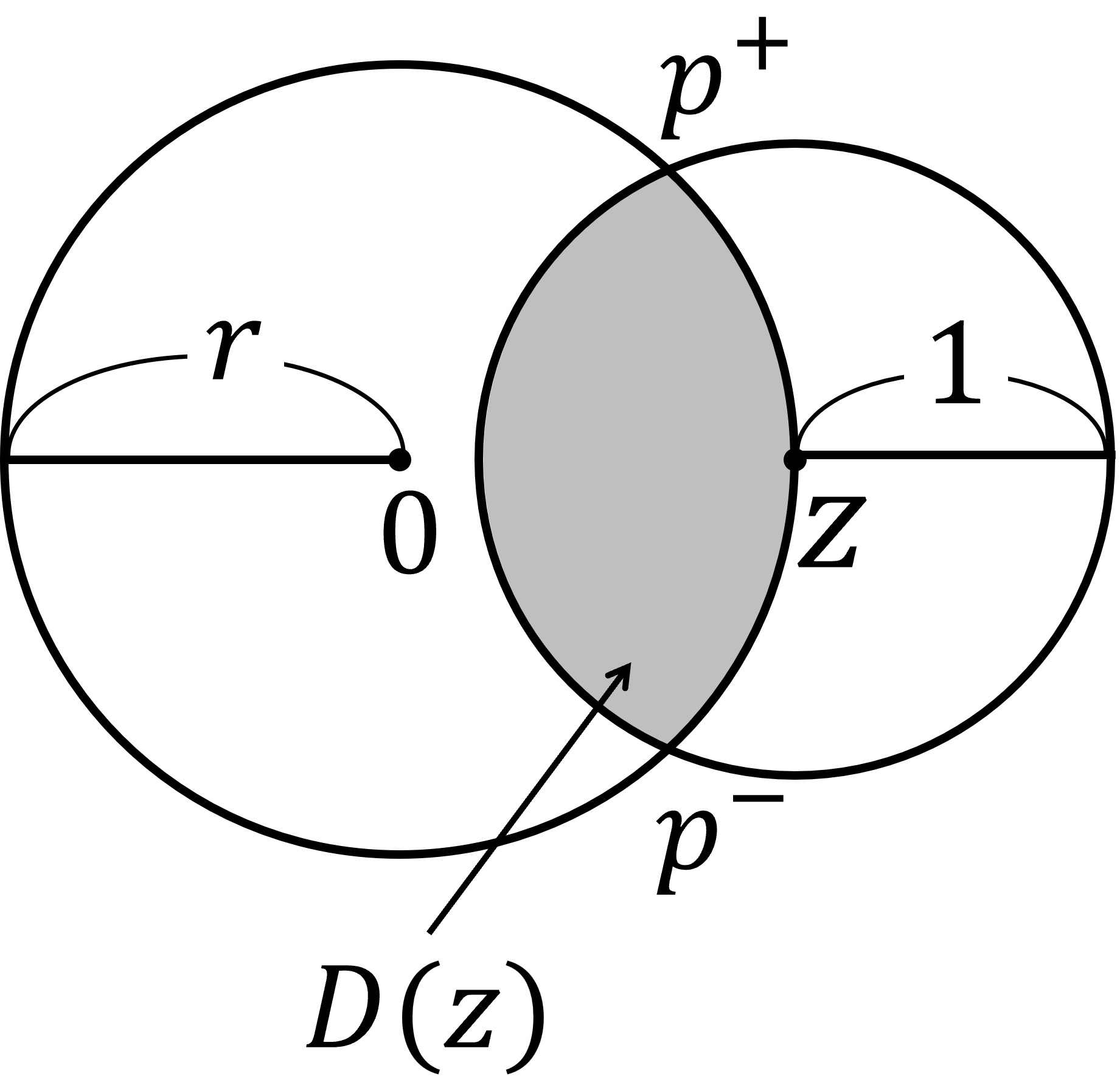}
\end{minipage}
\begin{minipage}{30mm}
\includegraphics[width=21mm]{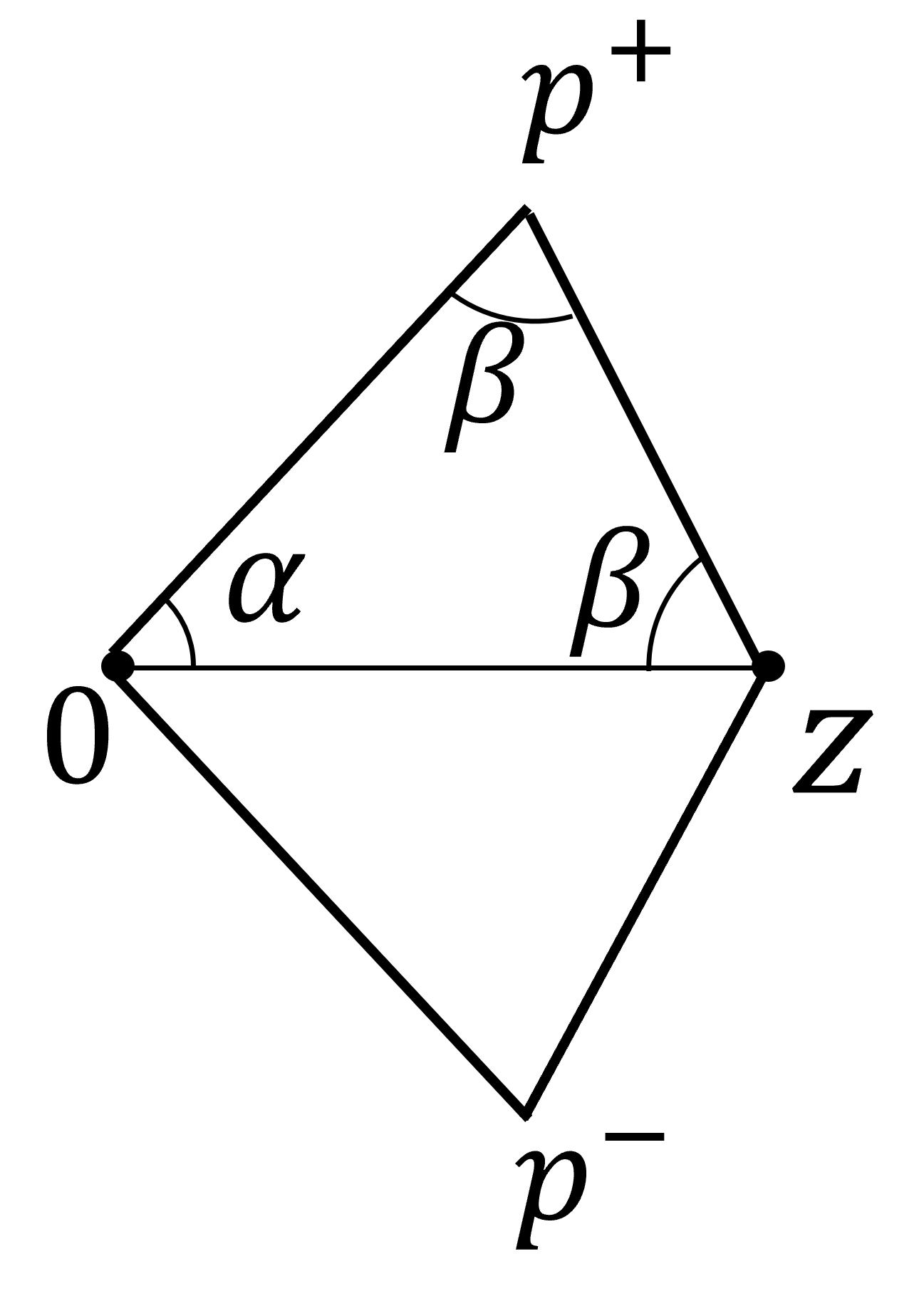}
\end{minipage}
\caption{{\itshape Case~2.} in the proof of {\itshape Example~\ref{ex:muPS}.}}
\end{center}
\end{figure}
We have
\begin{align*}
\mu \left( D(z)\right) &= \mu\left( S(0;p^{+}p^{-}) \right) + \mu\left( S(z;p^{+}p^{-})\right) - \mu \left( T(op^{-}zp^{+}) \right) \\
&= \int_{-\alpha}^{\alpha}\int_{0}^{r/2} 2 \sinh(2t) dtd\theta + \int_{-\beta}^{\beta}\int_0^{1/2} 2\sinh(2t)dtd\theta - 2\mu(\triangle ozp^{+}) \\
&= 2\alpha (\cosh{r} -1) +2\beta (\cosh{1}-1) -2(\pi- \alpha -2\beta) \\
&= 2\alpha \cosh{r} +2\beta (\cosh{1}+1)-2\pi .
\end{align*}
The third equality follows by Fact~\ref{fact:hyp}(2). By the first law of cosines (Fact~\ref{fact:hyp}(1)),
\begin{align*}
\cos\alpha &= \frac{\cosh^2{r}-\cosh{1}}{\sinh^2{r}},\\
\cos\beta &= \frac{\cosh{r}\cosh{1}-\cosh{r}}{\sinh{r}\sinh{1}}.
\end{align*}
Therefore, we have the following evaluation of $\partial c(z)$:
\begin{align*}
\partial c(z) &= 2\pi (\cosh{1}-1) -4\cos^{-1}\left( \frac{\cosh^2{r}-\cosh{1}}{\sinh^2{r}} \right)\cosh{r} \\
& \qquad -4 \cos^{-1}\left( \frac{\cosh{r}\cosh{1}-\cosh{r}}{\sinh{r}\sinh{1}} \right)(\cosh{1}+1) +4\pi \tag{**} \\
&\geq \lim_{r\to \infty} \left\{  2\pi (\cosh{1}-1) -4\cos^{-1}\left( \frac{\cosh^2{r}-\cosh{1}}{\sinh^2{r}} \right)\cosh{r} \right.\\
& \qquad \left. -4 \cos^{-1}\left( \frac{\cosh{r}\cosh{1}-\cosh{r}}{\sinh{r}\sinh{1}} \right)(\cosh{1}+1) +4\pi \right\} \\
&= \frac{(1+e)^2\tan^{-1}\left( \frac{e-1}{2\sqrt{e}}\right)}{2e} -\frac{e-1}{e} > 0 .
\end{align*}
The second inequality holds since the function (**) as $r$ is monotonically decreasing in $r>1/2$. We checked this, the third equality, and the last evaluation by calculation of computer.

By {\itshape Case~1} and {\itshape Case~2}, we have $\partial c (z) \geq \varepsilon$ for all $z \in D^2$, where \[ \varepsilon := \min \left\{ 2\pi (\cosh{1} - 1) -4\pi\left(\cosh{\frac{1}{2}} -1\right) ,\frac{(1+e)^2\tan^{-1}\left( \frac{e-1}{2\sqrt{e}}\right)}{2e} -\frac{e-1}{e} \right\} > 0.\] This is the desired result.
\qed

\section{Coarse property of $\mu$-PS}
In this section, we consider the following question:
\begin{question}
Is the existence of $\mu$-PS well-defined up to coarse equivalence?
\end{question}
For Ponzi schemes, the answer to this question is YES, in other words, the following fact is known  (proved in \cite{BW92} and \cite{Roe03}):
\begin{fact}{\rm (\cite[Prop.\ 3.35]{Roe03})}
If $X$ and $Y$ are coarsely equivalent, and $X$ admits a Ponzi scheme, then $Y$ also admits a Ponzi scheme.
\end{fact}
However, for $\mu$-PSs, the answer is NO. Indeed, a coarse space $X$ that has a $\mu$-PS is coarsely equivalent to itself with the zero measure $0$, but it does not have a $0$-PS. To avoid such situation, we  impose an additional condition to coarsely equivalent maps. In this section, we prove the following theorem:
\begin{theorem}\label{thm:muPScp}
Let $(X,\mathcal{E}^X), (Y, \mathcal{E}^Y)$ be measurable coarse spaces, $\mu_X$ and $\mu_Y$ be their uniform measures respectively, and $(X,\mu_X)$ and $(Y,\mu_Y)$ be $\sigma$-finite. Let a measurable map $\varphi\colon (X, \mathcal{E}^X, \mu_X) \to (Y, \mathcal{E}^Y, \mu_Y)$ be measure effectively proper and coarsely equivalent. If a $\mu_X$-1 chain $c$ on $X$ is a $\mu_X$-PS, then the induced $\mu_Y$-1 chain $\varphi_* c$ is a $\mu_Y$-PS.
\end{theorem}

We define the notion in Theorem~\ref{thm:muPScp}. Let $\varphi\colon(X,\mathcal{B}^X, \mu_X) \to (Y,\mathcal{B}^Y,\mu_Y)$ be a measurable map between measure spaces. Let $\varphi_*\mu_X$ denote the push-forward of $\mu_X$ by $\varphi$, i.e.,$(\varphi_*\mu_X)(B) := \mu_X(\varphi^{-1}(B))$ for all $B \in \mathcal{B}^Y$.
\begin{definition}
The map $\varphi$ is measure effectively proper if there exists a positive constant $C>0$ such that 
\[ (\varphi_* \mu_X)(B) \leq C \mu_Y(B)\]
for all $B \in \mathcal{B}^Y$.
\end{definition}

\begin{remark}
Let $\varphi\colon (X ,\mathcal{E}^X,\#_X) \to (Y, \mathcal{E}^Y, \#_Y)$ be a map between uniform locally finite coarse spaces with counting measures. If the map $\varphi$ is effectively proper, then it is measure effectively proper. Indeed, there exists $E_0 \in \mathcal{E}^X$ such that $(\varphi\times \varphi)^*(\Delta_Y) \subset E_0$. For any finite set $B =\{ b_1,\cdots, b_m\}$ in $Y$, we have 
\begin{align*}
(\varphi_*\#_X)(B) 
&= \#_X \varphi^{-1}\left( \bigcup_{i=1}^m b_i\right) \\
&= \#_X \varphi^{-1}\left( \bigcup_{i=1}^m (\Delta_Y)_{b_i}\right) \\
&\leq \sum_{i=1}^m \#_X \left( \varphi^{-1}(\Delta_Y)_{b_i} \right) \\
&\leq \sum_{i=1}^m \#_X (E_0)_{a_i}  \\
&\leq \left( \sup_{x \in X} \#_X (E_0)_x \right) \#_Y B ,
\end{align*}
where $a_i \in \varphi^{-1}(b_i)$. Therefore, $\varphi$ is measure effectively proper.
\end{remark}

We prove Theorem~\ref{thm:muPScp} step by step, showing first several propositions.
\begin{proposition}\label{prop:push}
Let $\varphi\colon (X,\mu_X) \to (Y,\mu_Y)$ be a measure effectively proper. Then, the following inequalities hold
\begin{flalign*}
&(1) \quad \int_Y |f| d(\varphi_* \mu_X) \leq C \int_Y |f| d\mu_Y & \\ 
&(2) \quad \int_{Y\times Y} |c| d(\varphi_*\mu_X\otimes\varphi_*\mu_X) \leq C^2 \int_{Y\times Y} |c| d\mu_Y\otimes\mu_Y &
\end{flalign*}
for any $f \in L^1(Y,\mu_Y)$ and $c \in L^1(Y\times Y ,\mu_Y\otimes\mu_Y )$.
\end{proposition}
\begin{proof}
(1) Let $f_n := \sum_i a^{(n)}_i \chi_{B^{(n)}_i}$ be a sequence of simple functions approximating $|f| \in L^1(Y,\mu_Y)$. We have 
\begin{eqnarray*}
\int_Y |f| d(\varphi_* \mu_X) 
&=& \lim_{n\to \infty} \int_Y \sum_i a_i^{(n)} \chi_{B^{(n)}_i} d(\varphi_*\mu_X) \\
&=& \lim_{n\to \infty} \sum_i a^{(n)}_i \mu_X\left( \varphi^{-1}(B^{(n)}_i) \right) \\
&\leq & \lim_{n\to \infty} \sum_i C a^{(n)}_i \mu_Y(B^{(n)}_i) \\
&=& C \lim_{n\to \infty} \int_Y f_n d\mu_Y \\
&=& C \int_Y |f| d\mu_Y .
\end{eqnarray*}
(2) By Toneli's theorem and (1), we have 
\begin{align*}
   \vspace{-10pt} &\int_{Y\times Y} |c(y_1,y_2)| d(\varphi_*\mu_x \otimes \varphi_*\mu_X)(y_1,y_2) \\
&= \int_Y \left(\int_Y |c(y_1,y_2)| d(\varphi_* \mu_X)(y_1)\right) d(\varphi_*\mu_X)(y_2) \\
&\leq  \int_Y C \left( \int_Y |c(y_1,y_2)| d\mu_Y(y_1)\right) d(\varphi_* \mu_X)(y_2) \\
&\leq C^2 \int_Y \left( \int_Y |c(y_1,y_2)| d\mu_Y(y_1) \right) d \mu_Y(y_2) \\
&= C^2 \int_{Y\times Y} |c(y_1,y_2)| d(\mu_Y \otimes \mu_Y)(y_1,y_2) 
\end{align*}
\qed
\end{proof}

For $f \in L^{\infty}(X,\mu_X), g \in L^1(X, \mu_X)$, a pairing $\langle f,g\rangle_X$ is defined by 
\[\langle f,g\rangle_X:= \int_X f(x)g(x) d\mu(x).\]

\begin{proposition}\label{prop:0chain}
Let $\varphi$ be measure effectively proper, and $f$ be a $\mu_X$-0 chain on $X$. Then the push-forward $\varphi_* f$ is a $\mu_Y$-0 chain, which is characterized by $\langle \varphi_* f, g\rangle_Y = \langle f,\varphi^* g\rangle_X$ for any $g \in L^1(Y,\mu_Y) $.
\end{proposition}
\begin{proof}
 We prove that $\varphi_* f$ is a bounded functional on $L^1(Y,\mu_Y)$. Then, the functional $\varphi_* f$ can be considered as a $L^{\infty}$ function on $Y$ by Fact~\ref{fact:mth}(3).

 For a $L^1$ function $g$ on $Y$, we have 
 \begin{eqnarray*}
|\langle \varphi_* f, g\rangle_Y| 
&\leq & \int_X |f(x)| |g\varphi (x)| d\mu_X(x) \\
&\leq& ||f||_{L^{\infty}(X)} \int_X |g \varphi(x)|d\mu_X(x) \\
&=& ||f||_{L^{\infty}(X)} \int_Y |g(y)|d(\varphi_*\mu_X)(y) \\
&\leq & C ||f||_{L^{\infty}(X)} ||g||_{L^1(Y)} .
 \end{eqnarray*}
 The last inequality follows by Proposition~\ref{prop:push}. Therefore, $\varphi_* f$ is well-defined as $L^{\infty}$ function on $Y$.
 \qed
\end{proof}

\begin{proposition}
Let $\varphi$ be measure effectively proper and bornologous, and c be a $\mu_X$-1 chain. Then, the push-forward $\varphi_* c$ is $\mu_Y$-1 chain on $Y$, which is characterized by $\langle \varphi_* c, g\rangle_{Y\times Y} = \langle c, \varphi^* g\rangle_{Y\times Y}$ for any $g \in L^1(Y\times Y, \mu_Y\otimes \mu_Y)$.
\end{proposition}
\begin{proof}
We can check that $\varphi_* c \in L^{\infty}(Y\times Y, \mu_Y\otimes \mu_Y)$ in the same way of the last proposition. We prove that there exists $F_c \in \mathcal{E}^Y \cap (\mathcal{B}^Y \times \mathcal{B}^Y)$ such that $\int_{Y\times Y \setminus F_c} |\varphi_* c| d\mu_Y\otimes \mu_Y = 0$.

Since the map $\varphi$ is bornologous, $(\varphi \times \varphi)_*(E_c) \in \mathcal{E}^Y$. We take $F_c \in \mathcal{E}^Y \cap (\mathcal{B}^Y \times \mathcal{B}^Y)$, containing $(\varphi \times \varphi)_*(E_c)$. It suffices that 
\[ \int_U |\varphi_* c| d\mu_Y\otimes \mu_Y = 0\]
for all $\mu_Y$ finite subsets $U$ in $Y\times Y \setminus F_c$, since the measure space $(Y,\mu_Y)$ is $\sigma$-finite. We have 
\begin{eqnarray*}
\vspace{-10pt} && \int_U |\varphi_* c| d\mu_Y \otimes \mu_Y \\
&=& \int_{Y\times Y} |\varphi_* c(y_1, y_2)| \chi_U (y_1,y_2) d\mu_Y\otimes\mu_Y (y_1,y_2) \\
&=& \int_{X\times X}|c(x_1,x_2)| \chi_U(\varphi(x_1),\varphi(x_2) ) d\mu_X\otimes\mu_X(x_1,x_2) \\
&=& \int_{X\times X \setminus E_c}|c(x_1,x_2)| \chi_U(\varphi(x_1),\varphi(x_2) ) d\mu_X\otimes\mu_X(x_1,x_2) \\
&&\qquad + \int_{E_c}|c(x_1,x_2)| \chi_U(\varphi(x_1),\varphi(x_2) ) d\mu_X\otimes\mu_X(x_1,x_2) .
\end{eqnarray*}
The first term vanishes by the assumption of $\mu_X$-1 chain, and the second term vanishes since $(\varphi\times \varphi)_* (E_c) \cap U = \emptyset$. Therefore, the bounded function $\varphi_* c$ on $Y\times Y$ is a $\mu_Y$-1 chain on $Y$.
\qed
\end{proof}

\begin{proposition} \label{prop:eff}
Let $\varphi$ be measure effectively proper and coarsely equivalent, and $f$ be an effective $\mu_X$-0 chain. Then the induced $\mu_Y$-0 chain $\varphi_* f$ is effective.
\end{proposition}
\begin{proof}
Let a measurable controlled set $E_f \in \mathcal{E}^X \cap (\mathcal{B}^X\times \mathcal{B}^X)$ satisfy 
\[ \int_{(E_f)_{x_0}} f(x) d\mu_X(x) \geq 1 ,\]
and $\psi\colon Y \to X$ be a coarsely inverse map, i.e.,  they satisfy $\phi\circ \varphi \sim id_X$, and $\varphi \circ \psi \sim id_Y$. We prove the following claim:
\begin{claim}
There exists $F_f \in \mathcal{E}^Y \cap (\mathcal{B}^Y\times \mathcal{B}^Y)$ such that it satisfies $(E_f)_{\psi(y)} \subset \varphi^{-1}((F_f)_y)$ for all $y \in Y$.
\end{claim}
\begin{proof}
Since $\varphi$ is bornologous, and $\varphi\circ\psi \sim id_Y$, we have 
\[ \{(y,\varphi\psi(y) : y\in Y\} \circ (\varphi\times \varphi)_* E_f \in \mathcal{E}^Y.\]
Thus, there exists $F_f \in \mathcal{E}^Y \cap (\mathcal{B}^Y\times \mathcal{B}^Y)$ containing the above set. We have 
\begin{align*}
\vspace{-10pt}&\left( \{(y^{\prime},\varphi\psi(y^{\prime}): y^{\prime} \in Y\}\circ (\varphi \times \varphi)_*E_f \right)_y \\
& = \{ y^{\prime} \in Y: (\varphi\psi(y), y^{\prime}) \in (\varphi \times \varphi)_* E_f\} \\
& \supset \{ \varphi(x) \in Y : (\psi(y), x) \in E_f \} =\varphi\left( (E_f)_{\psi(y)}\right).
\end{align*}
Hence, $\varphi\left( (E_f)_{\psi(y)}\right) \subset \left( F_f\right)_y.$ Therefore, we have the desired result.
\qed
\end{proof}

We have for all $y_0 \in Y$
\begin{align*}
\int_{(F_f)_{y_0}} (\varphi_* f)(y) d\mu_Y(y) 
&= \int_Y \varphi_* f(y) \chi_{(F_f)_{y_0}} (y) d\mu_Y(y) \\
&= \int_X f(x) \chi_{(F_f)_{y_0}} (\varphi(x)) d\mu_X(x) \\
&= \int_X f(x) \chi_{\varphi^{-1}\left((F_f)_{y_0}\right)}(x) d\mu_X(x) \\
&= \int_{\varphi^{-1} \left( (F_f)_{y_0} \right)} f(x)  d\mu_X(x) \\
&\geq \int_{(E_f)_{\psi(y_0)}} f(x) d\mu_X(x) \geq 1 .\\
\end{align*}
Therefore, the induced $\mu_Y$-0 chain $\varphi_* f$ is effective.
\qed
\end{proof}

\begin{proposition}\label{prop:del}
Let $\varphi$ be measure effectively proper and bornologous. Then, we have
$\partial \varphi_* = \varphi_* \partial$. More precisely, we have $\partial (\varphi_* c) = \varphi_* (\partial c)$ for a $\mu_Y$-1 chain $c$ on $Y$.
\end{proposition}
\begin{proof}
Assume that $E_c \in \mathcal{E}^X \cap (\mathcal{B}^X\times \mathcal{B}^X)$ and $F_c \in \mathcal{E}^Y \cap (\mathcal{B}^Y\times \mathcal{B}^Y)$ satisfy the following conditions:
\begin{align*}
(\varphi \times \varphi)_* E_c &\subset F_c ,\\
\int_{Y\times Y \setminus F_c} |\varphi_* c| d\mu_Y\otimes \mu_Y &= \int_{X\times X\setminus E_c} |c| d\mu_X\otimes \mu_X = 0.
\end{align*}
For a $L^1$ function $g$ on $Y$, we have
\begin{align*}
\langle \partial \varphi_*c, g\rangle_Y 
&= \int_Y \partial \varphi_* c (y_2) g(y_2) d\mu_Y(y_2) \\
&= \int_Y \left( \int_Y (\varphi_* c)(y_1,y_2) d\mu_Y (y_1) - \int_Y (\varphi_* c)(y_2,y_1) d\mu_Y(y_1)\right) d\mu_Y(y_2) \\
&= \int_{Y\times Y} \varphi_*(y_1,y_2) \chi_{F_c}(y_1,y_2)g(y_2) d\mu_Y\otimes \mu_Y(y_1,y_2) \\
&\quad - \int_{Y\times Y} \varphi_*(y_2,y_1) \chi_{F_c}(y_2,y_1)g(y_2) d\mu_Y\otimes \mu_Y(y_1,y_2) \\
&= \int_{X\times X} c(x_1,x_2) \chi_{F_c} \left(\varphi(x_1), \varphi(x_2)\right) g\varphi(x_2) d\mu_X\otimes\mu_X (x_1,x_2) \\
& \quad - \int_{X\times X} c(x_2,x_1) \chi_{F_c} \left(\varphi(x_2), \varphi(x_1)\right) g\varphi(x_2) d\mu_X\otimes\mu_X (x_1,x_2)\\
&= \int_{X\times X} c(x_1,x_2) \chi_{(\varphi\times \varphi)^* F_c} \left(x_1, x_2\right) g\varphi(x_2) d\mu_X\otimes\mu_X (x_1,x_2) \\
& \quad - \int_{X\times X} c(x_2,x_1) \chi_{(\varphi \times \varphi)^* F_c} \left(x_2, x_1\right) g\varphi(x_2) d\mu_X\otimes\mu_X (x_1,x_2) \\
&= \int_{X\times X} c(x_1,x_2)  g\varphi(x_2) d\mu_X\otimes\mu_X (x_1,x_2) \\
& \quad - \int_{X\times X} c(x_2,x_1)g\varphi(x_2) d\mu_X\otimes\mu_X (x_1,x_2) \\
&= \int_X \left( \int_X c(x_1,x_2) d\mu(x_1) - \int_X c(x_2,x_1) d\mu_X(x_1)\right) g\varphi(x_2) d\mu_X(x_2) \\
&= \int_X \partial c(x_2) g\varphi(x_2) d\mu_X(x_2) \\
&= \langle \partial c , \varphi^* g\rangle_X = \langle \varphi_* \partial c ,g \rangle_Y.
\end{align*}
The 4th equality holds since the function $\chi_{F_c} (y_1,y_2) g(y_2)$ is a $L^1$ function on $Y\times Y$, and the 6th equality holds since the $\mu_X$-1 chain $c$ vanishes outside $(\varphi\times \varphi)^* F_c \supset E_c$. Therefore, we have $\partial \varphi_* = \varphi_* \partial$.
\qed
\end{proof}

Combining Proposition~\ref{prop:0chain} - Proposition~\ref{prop:del}, we obtain Theorem~\ref{thm:muPScp}.\\
{\itshape Proof of Theorem~\ref{thm:muPScp}.}
We have $\partial (\varphi_* c) = \varphi_* (\partial c)$ by Proposition~\ref{prop:del}. Since $\partial c$ is effective, thus $\varphi_*(\partial c)$ is effective by Proposition~\ref{prop:eff}. Therefore, $\varphi_* c$ is a $\mu_Y$-Ponzi scheme.
\qed

\section{Relationship between Ponzi scheme and $\mu$-PS }
\subsection{Proof of Theorem~1.1}
Let $(X,\mathcal{B}, \mathcal{E})$ be a measurable coarse space, and $\mu$ be a uniform measure with respect to $\mathcal{E}$, and $(X,\mu)$ be a $\sigma$-finite measure space. In this subsection, we consider the following question:
\begin{question}
Does there exist a  Ponzi scheme if there exists a $\mu$-PS? Does there exist a $\mu$-PS if there exists a Ponzi scheme?
\end{question}
We give a partial answer to this question in Theorem~\ref{thm:muPS}. In order to state the theorem, we prepare some notions.

\begin{definition}
\begin{enumerate}
\item A subset $\Lambda$ in $X$ is a quasi lattice if it is uniform locally finite and there exists a controlled set $E_0 \in \mathcal{E}$ such that $X=E_0[\Lambda]$.
\item Let $S$ be a subset in $X$. A measurable controlled set $E \in \mathcal{E}\cap (\mathcal{B}\times \mathcal{B}) $ is constant on $S$ with respect to $\mu$ if there exists a positive constant $C>0$ such that $\mu(E_s) = C $ for all $s \in S$.
\end{enumerate}
\end{definition}

\begin{theorem}\label{thm:muPS}
\begin{enumerate}
\item Assume $X$ has a $\mu$-PS and a quasi lattice $\Lambda$. Then, $X$ has a Ponzi scheme.
\item Assume $X$ has a Ponzi scheme $\theta$ such that $\supp \theta \subset S \times S$ for a uniform locally finite subset $S$, and there exists a measurable controlled set $E$ which is constant on S with respect to $\mu$. Then, $X$ has a $\mu$-PS. 
\end{enumerate}
\end{theorem}

\begin{proof}
(1) There exists a measurable controlled set $E_0 \in \mathcal{E}\cap (\mathcal{B}\times \mathcal{B})$ such that $X = \bigcup_{\lambda \in \Lambda} (E_0)_\lambda$. We introduce a total order $\prec$ into $\Lambda$. Let  the  set $(\tilde{E_0})_\lambda$ be defined by
\begin{align*}
(\tilde{E_0})_\lambda 
&:= (E_0)_\lambda \setminus \bigcup_{s \prec \lambda} (E_0)_s \\
&= (E_0)_\lambda \setminus \bigcup \{ (E_0)_s : s \prec \lambda , (E_0)_s  \cap (E_0)_\lambda \neq \emptyset \}.
\end{align*}
The set $\{ s \in \Lambda : (E_0)_s \cap (E_0)_\lambda \neq \emptyset\}$ is finite, since such $ s\in \Lambda $ have to satisfy $ s \in (E_0\circ E_0^T)_\lambda$ and the quasi lattice $\Lambda$ is uniform locally finite. Thus, the set $(\tilde{E_0})_\lambda$ is measurable for all $\lambda \in \Lambda$, and we have $X = \bigsqcup_{\lambda \in \Lambda} (\tilde{E_0})_\lambda$ (the right-hand side stands for the disjoint union of all $(\tilde{E_0})_\lambda$).

The map $ p\colon X \to \Lambda $ defined by $p(x) = \lambda$ if $x \in (\tilde{E_0})_\lambda$ is measurable. Moreover, it is measure effectively proper from $(X,\mu_X)$ to $(X, \#|_\Lambda)$, since we have 
\begin{align*}
\mu (p^{-1}(B)) &= \mu \left( p^{-1}(\{ \lambda_1 ,\cdots, \lambda_m \})\right) \\
&= \mu \left( \bigsqcup_{i=1}^m (\tilde{E_0})_{\lambda_i} \right) \\
&= \sum_{i=1}^m \mu \left( (\tilde{E_0})_{\lambda_i}\right) \\
&\leq \left( \sup_{x\in X}\mu \left((E_0)_x\right) \right) \#|_\Lambda (B)
\end{align*}
for a subset $B=\{ \lambda_1, \cdots, \lambda_m\} \subset \Lambda$.

The map $p$ is coarsely surjective since $X= \bigcup_{x\in X} (E_0)_{p(x)}$. It is  bornologous and effectively proper, since we have
\begin{align*}
(p\times p)_*(E) &\subset E_0^T \circ E \circ E_0 \in \mathcal{E} \\
(p\times p)^*(E) &\subset E_0\circ E \circ E_0^T \in \mathcal{E}
\end{align*}
for any controlled set $E \in \mathcal{E}$. Thus, it is coarsely equivalent by Fact~\ref{fact:ceq}. Therefore, the coarse space $X$ has a Ponzi scheme by Theorem~\ref{thm:muPScp}.

(2) Let $\partial_X$ denote the boundary map from $C_1(X,\mu)$ to $C_0(X,\mu)$ and $\partial_S$ denote that from $C_1(X,\#|_S)$ to $C_0(X,\#|_S)$.  We consider a map $c\colon X \times X \to \bbbr$ defined by
\[ c(x,y) := \sum_{(s_1,s_2) \in S\times S} \theta(s_1,s_2) \chi_{E_{s_1}\times E_{s_2}}(x,y) . \]
The map $c$ is bounded since $S$ is a uniform locally finite subset in $X$. In order to analyze the function $c$, we prove the following equation.
\[ \partial_X c (x) = C \sum_{s\in S} \partial_S \theta (s) \chi_{E_s}(x) \]
Indeed, we have 
\begin{align*}
\partial_X c(x) 
&= \int_X c(y,x) d\mu(y) -\int_X c(x,y) d\mu(y) \\
&= \int_X  \left( \sum_{(s_1,s_2) \in S\times S} \theta (s_1,s_2) \chi_{E_{s_1}\times E_{s_2}}\right) (y,x) d\mu(y)\\
&\quad -\int_X  \left( \sum_{(s_1,s_2) \in S\times S} \theta (s_1,s_2) \chi_{E_{s_1}\times E_{s_2}}\right) (x,y) d\mu(y)\\
&= \sum_{s_2 \in E_x^T \cap S} \sum_{s_1 \in S} \int_X \theta(s_1,s_2) \chi_{E_{s_1}}(y) d\mu(y) \\
&\qquad - \sum_{s_1 \in E_x^T \cap S} \sum_{s_2 \in S} \int_X \theta(s_1,s_2) \chi_{E_{s_2}}(y) d\mu(y) \\
&= C \left( \sum_{s_2 \in E_x^T \cap S} \sum_{s_1 \in S} \theta (s_1,s_2) - \sum_{s_2 \in E_x^T \cap S} \sum_{s_1 \in S} \theta (s_2,s_1)\right) \\
&= C \left( \sum_{s \in E_x^T\cap S} \partial_S \theta (s) \right) \\
&= C \sum_{s\in S} \partial_S \theta (s) \chi_{E_s}(x) .
\end{align*}
By the above equation, we also have $\partial c \geq 0$.

Since the function $\theta$ is a Ponzi scheme, there exists $E^{\prime} \in \mathcal{E}$ such that 
\[ \sum_{s \in E^{\prime}_{x_0}\cap S} \partial_S \theta (s) \geq 1\]
for all $x_0 \in X$. Note that $E_s \subset (E^{\prime}\circ E)_{x_0} $ if $s \in E_{x_0}^{\prime}\cap S$. 
Let $\tilde{E}^{\prime}$ be a measurable controlled set containing $E^{\prime}\circ E$. We have 
\begin{align*}
\int_{(\tilde{E^{\prime}})_{x_0}} (\partial_X c)(x) d\mu(x) 
&= C \int_{(\tilde{E^{\prime}})_{x_0}} \sum_{s \in S} \partial_S \theta(s) \chi_{E_s} (x) d\mu(x) \\
&= C  \sum_{s \in S}\int_{(\tilde{E^{\prime}})_{x_0}}  \partial_S \theta(s) \chi_{E_s} (x) d\mu(x) \\
&\geq C \sum_{s\in E^{\prime}_{x_0}\cap S} \int_{E_s} \partial_S \theta (s) \chi_{E_s}(x) d\mu(x) \\
&= C^2 \sum_{s \in E^{\prime}_{x_0}\cap S} \partial_S \theta (s) \geq C^2 .
\end{align*}
Therefore, the $\mu$-1 chain $c/C^2$ is a $\mu$-PS.
\qed
\end{proof}

\subsection{Amenability of discrete groups and $\mu$-PS}
By using Theorem~\ref{thm:muPS}, we prove the relationship between the amenability of a discrete group $\Gamma$ and the existence of a $\mu$-PS on a space $X$ with respect to $\Gamma$-invariant measure $\mu$ as follows:
\begin{corollary}\label{cor:gpactmuPS}
Let a discrete group $\Gamma$ act properly and cocompactly on a locally compact second-countable Hausdorff space $X$, and $\mu$ be a $\Gamma$-invariant non zero regular measure on $X$. Then $\Gamma$ is non-amenable if and only if the coarse space $(X,\mathcal{E}^X_{\Gamma})$ has a $\mu$-PS.
\end{corollary}
We review the definition of amenability of discrete groups:
\begin{definition}{\rm (\cite[Appx.\ G]{BDV08})}
Let $G$ be a discrete group.
\begin{enumerate}
\item A linear functional $m$ on $l^{\infty}(G)$ is called a mean if it satisfies the following conditions:
\begin{itemize}
\item $m(f) \geq 0$, for all non negative function $f$.
\item $m(\chi)= 1$, where $\chi$ is the constant function whose value is $1$.
\end{itemize}
Moreover, the mean $m$ is called $G$-invariant if $m(g\cdot f) = m(f)$ for any $g\in G, f \in l^{\infty}(G)$, where $g\cdot f$ denotes the left-translation of $f$ by $g$.
\item The group $G$ is amenable if it admits a $G$-invariant mean.
\end{enumerate}
\end{definition}
In order to prove the corollary, we recall the following fact:
\begin{fact}{\rm (\cite[Thm.\ 3.53]{Roe03})}\label{fact:ame}
Let $\Gamma$ be a discrete group. The group $\Gamma$ is non-amenable if and only if the coarse space $(\Gamma,\mathcal{E}^{\Gamma}_R)$ admits a Ponzi scheme.
\end{fact}
{\itshape Proof of Corollary \ref{cor:gpactmuPS}.}
Assume that $\Gamma$ is non-amenable. We take a compact set $C_0$ in $X$ such that $\mu(C_0) > 0$, and $x_0 \in C_0$. If such a compact set does not exist, then we obtain $\mu(X) = 0$ since the space $X$ is $\sigma$-compact, which contradicts the assumption of the measure $\mu$. Let $S = \{ \gamma x_0 \in X : \gamma \in \Gamma \}$. For each $s \in S$. we choose $\gamma_s \in \Gamma$ such that $s=\gamma_s x_0$. We write $E_S := \bigcup_{s\in S} \{s\}\times \gamma_s C_0$. 
This is a measurable set since $E_S$ is closed by the properness of the $\Gamma$ action. Moreover, it is a controlled set since it is contained in $E_{\Gamma,C_0}$.
By Fact~\ref{fact:ame}, $\Gamma$ has a Ponzi scheme. Since the map
\[ (\Gamma,\mathcal{E}^{\Gamma}_R) \ni \gamma \mapsto \gamma x_0 \in  (X,\mathcal{E}^X_{\Gamma})\]
is coarsely equivalent, the coarse space $(X,\mathcal{E}^X_{\Gamma})$ has a Ponzi scheme with support $S$. We have 
\[ \mu\left( (E_S)_s\right) = \mu(\gamma_s C_0) = \mu(C_0).\]
Thus, the measurable controlled set $E_S$ is constant on $S$ with respect to $\mu$. By Theorem~\ref{thm:muPS}, the coarse space $(X,\mathcal{E}^X_{\Gamma})$ has a $\mu$-PS.

Assume the contrary. We take a compact set $C_1$ in $X$ such that $X=\Gamma C_1$, and fix $x_1 \in C_1$. Let $\Lambda = \{ \gamma x_1 : \gamma \in \Gamma \}$. We claim that $\Lambda$ is a quasi lattice. Indeed, we have 
\[ E_{\Gamma,C_1}[\Lambda] = \bigcup_{\gamma \in \Gamma} (E_{\Gamma,C_1})_{\gamma x_1} \supset \bigcup_{\gamma \in \Gamma} \gamma C_1 = X  .\]
Let a compact set $C \subset X$ satisfy $X=\Gamma C$. We take $\gamma_x \in \Gamma$ for each $x \in X$ such that $x \in \gamma_x C$. We have 
\begin{align*}
    \# (E_{\Gamma,C})_x \cap \Lambda &\leq \# \{ \gamma \in \Gamma : \gamma x_1 \in \gamma_x \Gamma_C C\} \\
    &= \# \{ \gamma \in \Gamma : \gamma x_1 \in \Gamma_C C \}
\end{align*}
The right-hand side is independent of $x$ and finite since the $\Gamma$-action is proper. Thus, $\Lambda$ is uniform locally finite. Therefore, $\Lambda$ is a quasi lattice. By Theorem~\ref{thm:muPS}, the coarse space $(X,\mathcal{E}^X_{\Gamma})$ has a Ponzi scheme. Since it is coarsely equivalent to $(\Gamma ,\mathcal{E}^{\Gamma}_R)$, $(\Gamma,\mathcal{E}^{\Gamma}_R)$ also has a Ponzi scheme. By Fact~\ref{fact:ame}, the discrete group $\Gamma$ is non-amenable. 
\qed

\subsubsection{Acknowledgements} The author expresses sincere gratitude to the organizers, Professor Ali Baklouti and Professor Hideyuki Ishi, for their warm hospitality during the 7th Tunisian Japanese conference in honor of Professor Toshiyuki Kobayashi held in Monastir (Tunisia). 

The author express sincere gratitude to Professor Toshiyuki Kobayashi for his deep insights, generous support, and constant encouragement. The author also would like to thank Takayuki Okuda for providing him solid advice for his study, V\'ictor P\'erez-Vald\'es for correcting his writing, and Federico Vigolo for providing valuable comments on his preprint.


\providecommand{\bysame}{\leavevmode\hbox to3em{\hrulefill}\thinspace}
\providecommand{\MR}{\relax\ifhmode\unskip\space\fi MR }
\providecommand{\MRhref}[2]{%
  \href{http://www.ams.org/mathscinet-getitem?mr=#1}{#2}
}
\providecommand{\href}[2]{#2}

\bibliographystyle{myamsalpha}

\end{document}